\documentclass[journal]{IEEEtran}

\usepackage{mathptmx}       
\usepackage{helvet}         
\usepackage{courier}        
\usepackage{type1cm}        
\usepackage{makeidx}         
\usepackage{graphicx}        
         
\usepackage{multicol}        
\usepackage[bottom]{footmisc}
\usepackage{ifthen}
\usepackage{subfigure}
\usepackage[usenames,dvipsnames]{color}
			
\makeindex             

\usepackage{graphics} 
\usepackage{graphicx} 
\usepackage{epsfig} 
\usepackage{mathptmx} 
\usepackage{times} 
\usepackage{mathtools}  
\usepackage{amssymb}  
\usepackage{amsthm}

\usepackage{amsfonts}

\usepackage{subfig}
\usepackage{verbatim}
\usepackage{comment}
\usepackage{algorithm}
\usepackage{algorithmic}
\usepackage{multirow}

\DeclareGraphicsExtensions{.pdf,.png,.jpg,.eps}

\def\ba{\begin{array}}
\def\ea{\end{array}}
\newcommand{\beq}{\begin{equation}}
\newcommand{\eeq}{\end{equation}}
\newcommand{\bq}{\begin{eqnarray}}
\newcommand{\eq}{\end{eqnarray}}
\newcommand{\bqn}{\begin{eqnarray*}}
\newcommand{\eqn}{\end{eqnarray*}}
\newcommand{\bee}{\begin{enumerate}}
\newcommand{\eee}{\end{enumerate}}
\newcommand{\bi}{\begin{itemize}}
\newcommand{\ei}{\end{itemize}}
\newcommand{\ii}{\mathbf{i}}

\newcommand{\re}{{\mathrm{Re}}}

\newcommand{\lin}{{\rm lin}}
\newcommand{\sh}{{\rm sh}}

\newboolean{showcomments}
\setboolean{showcomments}{true}
\newcommand{\slow}[1]{\ifthenelse{\boolean{showcomments}}
{ \textcolor{red}{(SL:  #1)}}{}}
\newcommand{\fz}[1]{\ifthenelse{\boolean{showcomments}}
{ \textcolor{blue}{#1}}{}}

\theoremstyle{definition}
\newtheorem{theorem}{Theorem}
\newtheorem{lemma}{Lemma}

\theoremstyle{definition}

\newtheorem{remark}{Remark}

\begin{document}

\title{A Note on Branch Flow Models with Line Shunts}

\author{Fengyu~Zhou,~\IEEEmembership{Student Member,~IEEE,}
        and~Steven~H.~Low,~\IEEEmembership{Fellow,~IEEE}
\thanks{Zhou and Low are with the Electrical Engineering Department, California Institute of Technology, Pasadena, CA 91125 USA (e-mails: \{f.zhou; slow\}@caltech.edu)}
}

%
%


\maketitle

\begin{abstract}
When the shunt elements in the $\Pi$ circuit line model are assumed zero, it has been proved
that branch flow models are equivalent to bus injection models and that the second-order cone 
relaxation of optimal power flow problems on a radial network is exact under certain conditions. 
In this note we propose a branch flow model that includes nonzero line shunts and prove that 
the equivalence and the exactness of relaxation continue to hold under essentially the same 
conditions as for zero shunt elements.
\end{abstract}


\section{Introduction}

The DistFlow model is introduced in \cite{Baran1989a, Baran1989b} for radial networks.  
It is extended in \cite{Farivar-2013-BFM-TPS, Low2014a} to a branch flow model for 
general networks that may contain cycles.  The branch flow model is justified in \cite{Bose-2015-BFMe-TAC}
by proving its equivalence to the widely used bus injection AC power flow model.   For radial networks
the DistFlow equations are nonlinear and hence optimal power flow (OPF) problems formulated based on 
these equations are nonconvex.  Second-order cone programming (SOCP) relaxation is introduced in 
\cite{Farivar-2013-BFM-TPS} and a sufficient condition for the relaxation to be exact is proved there for 
radial networks.

All of these branch flow models in \cite{Baran1989a, Baran1989b, Farivar-2013-BFM-TPS, Bose-2015-BFMe-TAC, Low2014a}
 assume zero shunt elements (line charging) in the $\Pi$ circuit line model.
A branch flow model is recently proposed in \cite{Christakou2016} that includes
nonzero line shunts and a local algorithm is provided to compute a local optimal of OPF problems.  
Another branch flow model with nonzero line shunts is proposed in \cite{NickPaolone2018}.  Instead
of SOCP relaxation of the original OPF, \cite{NickPaolone2018} proposes to solve a more conservative 
approximation and provides a sufficient condition for the SOCP relaxation of the approximate OPF to be exact.

The purpose of this note is to propose an alternative branch flow model 
(equations \eqref{eq:bfm.10} for radial networks and \eqref{eq:bfm.9} for general networks)
that includes nonzero line shunts
(line charging in the $\Pi$ circuit model),
and prove that the equivalence and the exactness of second-order cone relaxation continue to hold under
essentially the same conditions as when line shunts are zero.  

This note only serves to complete the model and corresponding results in \cite{Farivar-2013-BFM-TPS, Bose-2015-BFMe-TAC}
by including line shunts.  It is \emph{not} our purpose to assess or extend the various advances since their publication.

\section{Background}

\noindent\emph{Notations.}
Let $\mathbb C$ denote the set of complex numbers and $\mathbb R$ the set of real numbers.
Let $\ii := \sqrt{-1}$.
For any $a\in\mathbb C$, $a^*$ denotes its complex conjugate.  
Unless otherwise specified a quantity $a$ denotes a vector whose $j$th entry is $a_j$, e.g., 
$s := (s_j, j\in N)$, $S := (S_{jk}, (j,k) \in L)$.
For real vectors $a$ and $b$, $a\leq b$ means $a_j\leq b_j$, $\forall j$.

Consider a single-phase power network with $N$ buses and $L$ lines modeled as a connected undirected 
graph $G(N, L)$ where $N := \{1, 2, \ldots, N \}$ and $L \subseteq  N \times  N$.\footnote{For simplicity we 
overload notations to use $N$ ($L$) to denote both the set and the number of buses (lines).   We also
 use $s_j$ and $S_{jk}$ below both as complex and real quantities interchangeably.  The meaning should be clear
from the context.}
This may represent the positive sequence network of a balanced three-phase system.
A line $(j, k)\in L$, or $j\sim k$, is represented by a $\Pi$ circuit model $(y^s_{jk}, y^m_{jk}, y^m_{kj})$ where 
$y^s_{jk} = y^s_{kj} \in\mathbb C$ are the series admittances of the line, $y^m_{jk} \in\mathbb C$ is the shunt 
admittance of the line at bus $j$, and $y^m_{kj}$ is the shunt admittance of the line at bus $k$.
Since the line model may include not just transmission lines, but other devices such as transformers, we do not
require $y^m_{jk}$ and $y^m_{kj}$ to be equal.

Let
\bi
\item $s_j := p_j + \ii q_j$ or $s_j := (p_j, q_j)$ denote the real and reactive power injections at bus $j$, either
	as a complex number or a pair of real numbers, $j\in N$.
\item $V_j$ denote the voltage phasor at bus $j$, $j\in N$.   
\ei
Without loss of generality, we fix bus $1$ as the slack bus and set $V_1 := 1\angle 0^\circ$.

A widely used power flow model, we called a bus injection model (BIM), is specified by the following set of equations 
that relates nodal injections $s$ and voltages $V$: for $j\in N$,
\bq
\!\!\!\!\!\!\!\!\!
s_j &\!\!\! = \!\!\! & \sum_{k: j\sim k} \left(y_{jk}^s\right)^* \!\! \left( |V_j|^2 - V_jV_k^* \right) \ + 
	\sum_{k: j\sim k} \left(y_{jk}^m\right)^* |V_j|^2
\label{eq:bim.1}
\eq
Let $\mathbb V$ denote the set of power flow solutions:
\bq
\mathbb V  :=  \{\ (s, V)\in\mathbb C^{2N} \, : \, (s, V) \text{ satisfies } \eqref{eq:bim.1},\ V_1 := 1\angle 0^\circ \ \}
\label{eq:bimsoln.1}
\eq

While the bus injection model \eqref{eq:bim.1} involves only nodal variables, branch flow models also involve 
branch variables.  Let
\bi
\item $I_{jk}$ denotes the phasor of the \emph{sending-end} current from bus $j$ to bus $k$, $(j,k)\in L$.
\item $S_{jk} := P_{jk} + \ii Q_{jk}$ or $S_{jk} := (P_{jk}, Q_{jk})$ denote the \emph{sending-end} real and 
	reactive power flows from bus $j$ to bus $k$, $(j,k)\in L$.
\ei

\vspace{0.1in}\noindent\textbf{DistFlow equations.}
When the shunt elements are assumed zero $y^m_{jk} = y^m_{kj} = 0$,  \cite{Baran1989a, Baran1989b} 
introduce a power flow model, called DistFlow equations, for single-phase radial networks 
(i.e., networks with a tree topology) and uses it to optimize 
the placement and sizing of switched capacitors in distribution circuits for volt/var control.  
DistFlow equations adopt a directed graph.
When the line is directed from bus $j$ to bus $k$, we refer to it by $j\rightarrow k \in L$.  
For each directed line $j\rightarrow k$, 
only the variable $S_{jk}$ is defined as the sending-end power flow from buses $j$ to $k$, but not the variable $S_{kj}$.  
The orientation of the graph can be arbitrary, but for ease of exposition, we sometimes assume without 
loss of generality that bus 1 is the root of the graph and every line points away from the root.  
Then the DistFlow equations are given by \cite{Baran1989a, Baran1989b}:\footnote{Here, the complex 
notation is used only as a shorthand for equations in the real domain, e.g., \eqref{eq:df.1b} denotes 
\bqn
v_j - v_k & = & 2\, \left(r_{jk}P_{jk} + x_{jk}Q_{jk}\right) - (r_{jk}^2+x_{jk}^2) \ell_{jk}
\eqn
}
\begin{subequations}
\bq
s_j & = & \sum_{k: j\rightarrow k} S_{jk} \ - \ \left( S_{ij} - z_{ij} \ell_{ij} \right),
	 \quad j \in  N
\label{eq:df.1a}
\\
v_j \ell_{jk} &  = & |S_{jk}|^2, 	 \qquad\qquad\qquad\quad\ \ \,  j\rightarrow k \in  L
\label{eq:df.1c}
\\
v_j - v_k &  = & 2\, \text{Re} \left(z_{jk}^* S_{jk} \right) - |z_{jk}|^2 \ell_{jk}, 
	\ \  j\rightarrow k \in L
\label{eq:df.1b}
\eq
\label{eq:df.1}
\end{subequations}
where $j$ is the parent of $k$ and 
 $z_{jk} := 1/y_{jk}^s$ is the series impedance of line $j\rightarrow k$.
Here the variable $v_j$ represents the squared voltage magnitude at bus $j$ and $\ell_{jk}$ represents the squared
current magnitude on line $j\rightarrow k$.
The key feature of the DistFlow equations is that they do not involve angles of voltage and current phasors
because these angles can be deduced from a solution $x := (s, v, \ell, S) \in\mathbb R^{3(N+L)}$ of the DistFlow 
equations \eqref{eq:df.1} when the network is radial, but not when the network contains cycles.

\vspace{0.1in}\noindent\textbf{Branch flow model.}
The model \eqref{eq:df.1} is extended to a branch flow model (BFM) in \cite{Farivar-2013-BFM-TPS, Low2014a}
for general networks that may
contain cycles by introducing a cycle condition.  Given any $x := (s, v, \ell, S)$ define 
\bq
\beta_{jk}(x) & := & \angle \left( v_j \ - \ z_{jk}^* S_{jk} \right), \quad j\rightarrow k \in L
\label{eq:defbeta.1}
\eq
and let $\beta(x) := (\beta_{jk}(x), j\rightarrow k \in L)$.  Even though a solution $x$ of the DistFlow equations
\eqref{eq:df.1} does not contain voltage phase angles, $\beta_{jk}(x)$ can be interpreted as the angle difference
across line $j\rightarrow k$ (also see Section \ref{sec:bfmws} below).
The DistFlow equations are extended in \cite{Farivar-2013-BFM-TPS} to general networks as:
\bq
\eqref{eq:df.1a}\eqref{eq:df.1c}\eqref{eq:df.1b}, \ \
\exists \theta\in\mathbb R^{N}\  \text{s.t.} \
\beta(x) \ = \ C^T\theta
\label{eq:df.2}
\eq
where $C$ is the $N \times L$ incidence matrix of the directed graph with
$C_{jl} = 1$ if $l = (j\rightarrow k)$ for some $k$, $-1$ if $l=i\rightarrow j$ for some $i$, and 0 otherwise. 
We refer to the condition $\beta(x)=C^T\theta$ on $x$ in \eqref{eq:df.2} as the \emph{cycle condition}.
For general networks the DistFlow equations \eqref{eq:df.1} can thus be interpreted as a relaxation of 
the branch flow model \eqref{eq:df.2} where the cycle condition is ignored.  
When a network is radial the cycle condition is vacuous and \eqref{eq:df.2} reduces to \eqref{eq:df.1}.

\vspace{0.1in}\noindent\textbf{Equivalence and SOCP relaxation.}
A justification of the branch flow model \eqref{eq:df.2}, and hence \eqref{eq:df.1}, is that they are equivalent 
to the widely used bus injection model \eqref{eq:bim.1}.  Specifically it is proved in \cite{Bose-2015-BFMe-TAC}
that, when shunt elements are assumed zero $y^m_{jk} = y^m_{kj} = 0$,
there is a bijection between the set $\mathbb V$ of solutions $(s, V)\in \mathbb C^{2N}$ and the set of solutions
$x\in\mathbb R^{3(N+L)}$ of \eqref{eq:df.2} for general networks and the set of solutions $x$ of
\eqref{eq:df.1} for radial networks.

The DistFlow equations \eqref{eq:df.1a}\eqref{eq:df.1b} are linear in its variable $x$
but \eqref{eq:df.1c} is quadratic.  Optimal power flow (OPF) problems formulated using the DistFlow
equations are therefore nonconvex.  Second-order cone program (SOCP) relaxation is introduced in
\cite{Farivar-2013-BFM-TPS} where \eqref{eq:df.1c} is relaxed to the convex constraint:
\bq
v_j \ell_{jk} &  \geq & |S_{jk}|^2, 	 \quad j\rightarrow k \in  L
\label{eq:soc.1}
\eq
The SOCP relaxation of an OPF problem is called exact if every solution of the relaxation
attains equality in \eqref{eq:soc.1} and hence is optimal for the original OPF problem.
It is proved in \cite{Farivar-2013-BFM-TPS} that, for radial networks, SOCP relaxation is exact if the
injections are not lower bounded.

\begin{remark}
The models in \cite{Baran1989a, Baran1989b} and \cite{Farivar-2013-BFM-TPS, Low2014a}
allow a nodal
shunt element at a bus $j$.  This can either be modeled as an injection $s_j^{\sh}$ as done in 
\cite{Baran1989a, Baran1989b} for volt/var control, in which case \eqref{eq:df.1a} is modified to:
\bqn
s_j \ + \ s_j^{\sh} & = & \sum_{k: j\rightarrow k} S_{jk} \ - \ \left( S_{ij} - z_{ij} \ell_{ij} \right)
\eqn
or as an admittance load $y^{\sh}_j$ as done in \cite{Farivar-2013-BFM-TPS}, in which case 
\eqref{eq:df.1a} is modified to:
\bqn
s_j  & = & \sum_{k: j\rightarrow k} S_{jk} \ - \ \left( S_{ij} - z_{ij} \ell_{ij} \right) \ + \ y_j^{\sh}v_j
\eqn
In this note we ignore these nodal shunt elements for simplicity, but all results will hold with 
straightforward modifications when they are included.
\qed
\end{remark}

The models in \cite{Baran1989a, Baran1989b} and \cite{Farivar-2013-BFM-TPS, Low2014a}, 
however,
assume zero shunt elements (line charging) $y^m_{jk} = y^m_{kj} = 0$ in the $\Pi$ circuit model.  
In the rest of this note we 
generalize the branch flow models \eqref{eq:df.1} and \eqref{eq:df.2} to include line shunts,
i.e., when  $y^m_{jk}$ and $y^m_{kj}$ are nonzero, and show that the results in \cite{Farivar-2013-BFM-TPS}
\cite{Bose-2015-BFMe-TAC} on equivalence and exact SOCP relaxation continue to hold under
essentially the same conditions as when $y^m_{jk} = y^m_{kj} = 0$.

\section{Branch flow model with line shunts}
\label{sec:bfmws}

Following \cite{Christakou2016} we consider the following branch flow model with line shunt elements
in complex form (adopting an undirected graph):
\begin{subequations}
\bq
s_j & = & \sum_{k: j\sim k} S_{jk}, \qquad  j \in N
\label{eq:bfm.7a}
\\
S_{jk}  & = & V_j \, I_{jk}^*,   \quad   S_{kj}  \ \ = \ \ V_k \, I_{kj}^*, 
\quad\ \ (j,k)\in L
\label{eq:bfm.7b}
\\
I_{jk}  & = &  y^s_{jk} (V_j - V_k) \ + \ y^m_{jk}V_j, 
\qquad (j,k)\in L
\label{eq:bfm.7c}
\\ 
I_{kj}  & = & y^s_{kj} (V_k - V_j) \ + \ y^m_{kj}V_k,  
\qquad (j,k)\in L
\label{eq:bfm.7d}
\eq
\label{eq:bfm.7}
\end{subequations}
where \eqref{eq:bfm.7a} imposes power balance at each bus, \eqref{eq:bfm.7b} defines branch power 
in terms of the associated voltage and current, and  \eqref{eq:bfm.7c}\eqref{eq:bfm.7d} describes the 
Kirchhoff's laws.
The main difference from the model in \cite{Farivar-2013-BFM-TPS} is the use of undirected rather
than directed graph when shunt elements are included so that line currents and power flows are defined
in both directions. 
Note that with shunt element, $S_{jk}$ and $\ell_{jk}$ are defined as the power flow and squared current at the terminals as shown in Fig \ref{fig:model}.

\begin{figure}
\centering
\includegraphics[width=0.9\columnwidth]{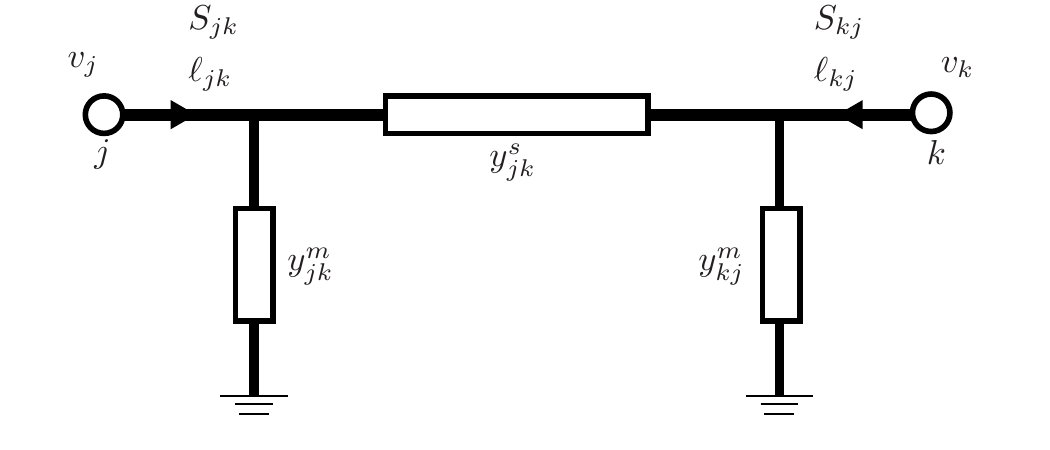}
\caption{The $\Pi$ circuit model, line parameters, and variables in BFMs.
}
\label{fig:model}
\end{figure}

Following \cite{Christakou2016}, we define
\bqn
\alpha_{jk} & := &  1 + z_{jk}^s\, y^m_{jk}, \qquad (j,k)\in L
\\
\alpha_{kj} & := &  1 + z_{kj}^s\, y^m_{kj}, \qquad (j,k)\in L
\eqn
where $z_{jk}^s := \left( y_{jk}^s\ \right)^{-1} = \left( y_{kj}^s \right)^{-1} =: z_{kj}^s$.  
Note that $\alpha_{jk} = \alpha_{kj}$ if and only if $y_{jk}^m = y_{kj}^m$ and 
$\alpha_{jk} = \alpha_{kj} = 1$ if and only if $y_{jk}^m = y_{kj}^m = 0$ as $|z^s_{jk}|\neq 0$.
Since (from \eqref{eq:bfm.7b}\eqref{eq:bfm.7c})
\bqn
S_{jk} & = & \left( y_{jk}^s + y_{jk}^m \right)^* |V_j|^2 \ - \ \left( y_{jk}^s \right)^*\, V_j\, V_k^*
\eqn
We have 
\begin{subequations}
\bq
V_j\, V_k^* & = & \alpha_{jk}^*\, |V_j|^2 \ - \ \left( z_{jk}^s \right)^*\, S_{jk}
\label{eq:VjVk*.1a}
\eq
Similarly we have
\bq
V_k\, V_j^* & = & \alpha_{kj}^*\, |V_k|^2 \ - \ \left( z_{kj}^s \right)^*\, S_{kj}
\label{eq:VjVk*.1b}
\eq
\label{eq:VjVk*.1}
\end{subequations}
This motivates the following generalization of \eqref{eq:df.1} as a branch flow model for radial networks
with shunt elements (we sometimes write $z_{jk}$ and $y_{jk}$ in place of $z_{jk}^s$ and $y_{jk}^s$ respectively
when there is no confusion):
for all $j\in N$ and $(j,k)\in L$,
\begin{subequations}
\bq
s_j & \!\!\!\!  =  \!\!\!\! & \sum_{k: j\sim k} S_{jk}
\label{eq:bfm.10a}
\\
v_j \, \ell_{jk}  &  \!\!\!\!  =  \!\!\!\!  & \left| S_{jk}\right|^2, \quad
v_k \, \ell_{kj}  \ = \ \left| S_{kj}\right|^2, \quad
\label{eq:bfm.10d}
\\
\!\!\!\!\!\!\!\!
\left|\alpha_{jk}\right|^2 v_j - v_k &  \!\!\!\!  =  \!\!\!\!  & 
		2\, \text{Re} \left(\alpha_{jk}\, z_{jk}^* \, S_{jk}\right) \ - \ \left| z_{jk} \right|^2 \ell_{jk}
\label{eq:bfm.10b}
\\
\!\!\!\!\!\!\!\!
\left|\alpha_{kj}\right|^2 v_k - v_j &  \!\!\!\!  =  \!\!\!\!  & 
		2\, \text{Re} \left(\alpha_{kj}\, z_{kj}^* \, S_{kj}\right) \ - \ \left| z_{kj} \right|^2 \ell_{kj}
\label{eq:bfm.10c}
\\
\!\!\!\!\!\!\!\!
\alpha_{jk}^*\, v_j \ - \  z_{jk}^*\, S_{jk} &  \!\!\!\!  =  \!\!\!\!  & \left( \alpha_{kj}^*\, v_k \ - \ z_{kj}^*\, S_{kj} \right)^*
\label{eq:bfm.10e}
\eq
\label{eq:bfm.10}
\end{subequations}

\begin{lemma}\label{lem:Svl}
When \eqref{eq:bfm.10b}\eqref{eq:bfm.10c}\eqref{eq:bfm.10e} hold, we have
\begin{align}\label{eq:Svl}
|S_{jk}|^2-v_j \ell_{jk} = |S_{kj}|^2-v_k \ell_{kj}.
\end{align}
\end{lemma}
\begin{proof}
Taking the squared magnitude on both sides of \eqref{eq:bfm.10e}, we obtain
\begin{align}\label{eq:mag}
\nonumber
&|\alpha_{jk}|^2v_j^2 + |z_{jk}|^2|S_{jk}|^2-2\re(\alpha_{jk}z_{jk}^*S_{jk})v_j\\
=&|\alpha_{kj}|^2v_k^2 + |z_{kj}|^2|S_{kj}|^2-2\re(\alpha_{kj}z_{kj}^*S_{kj})v_k
\end{align}
Together with \eqref{eq:bfm.10b} and \eqref{eq:bfm.10c}, we obtain \eqref{eq:Svl}.
\end{proof}

Lemma \ref{lem:Svl} shows that for radial networks, when we have \eqref{eq:bfm.10b}\eqref{eq:bfm.10c}\eqref{eq:bfm.10e} as constraints, 
only one direction in \eqref{eq:bfm.10d} is required to fully describe power flow equations. 
In practice, one could use either direction for line $j\sim k$ to simply the computation.

For general networks that may contain cycles, \eqref{eq:VjVk*.1} suggests extending 
the definition \eqref{eq:defbeta.1} of $\beta(x)$ to both directions of a line as follows:
\begin{subequations}
\bq
\!\!\!\!\!\!\!
\beta_{jk}( x) &\!\!\! := \!\!\! & \angle \left( \alpha_{jk}^*\, v_j -  z_{jk}^*\, S_{jk} \right),  \ \   (j, k)\in L
\label{ch:npfm; eq:Defbeta.1a}
\\
\!\!\!\!\!\!\!
\beta_{kj}( x) &\!\!\! := \!\!\! & \angle \left( \alpha_{kj}^*\, v_k -  z_{jk}^*\, S_{kj} \right), \ \   (j, k)\in L
\label{ch:npfm; eq:Defbeta.1b}
\eq
\label{ch:npfm; eq:Defbeta.1}
\end{subequations}
To generalize \eqref{eq:df.2}, fix an arbitrary graph orientation represented by the incidence 
matrix $C$ and label the lines such that entries $\beta_{jk}(x)$ in the first half of the vector 
$\beta(x) := (\beta_{jk}(x), \beta_{kj}(x), j\rightarrow k\in L)$ correspond to lines in the same direction as
specified in $C$, and entries $\beta_{kj}(x)$ in the second half of $\beta(x)$ correspond to
lines in the opposite direction.
Then we propose the following generalization of \eqref{eq:df.2} as a branch flow model for 
general networks with shunt elements: for all $j\in N$ and $(j,k)\in L$,
\begin{subequations}
\bq
& &
\eqref{eq:bfm.10a}\eqref{eq:bfm.10d}\eqref{eq:bfm.10b}\eqref{eq:bfm.10c}
\label{eq:bfm.9a}
\\
&& \exists \theta\in\mathbb R^{N} \ \ \text{s.t.} \ \
	\beta(x) \ = \ \begin{bmatrix} C^T \\ - C^T \end{bmatrix} \theta
\label{eq:bfm.9b}
\eq
\label{eq:bfm.9}
\end{subequations}
where the components $(\beta_{jk}(x), \beta_{kj}(x))$ of the vector $\beta(x)$ are ordered
such that \eqref{eq:bfm.9b} implies 
\bqn
\beta_{jk}(x) & = & \theta_j - \theta_k \ \ = \ \  - \beta_{kj}(x).
\eqn
As \eqref{eq:bfm.10d}\eqref{eq:bfm.10b}\eqref{eq:bfm.10c} imply \eqref{eq:mag}, they also imply that
the magnitudes are equal on both sides of \eqref{eq:bfm.10e}.
Hence \eqref{eq:bfm.10e} {can be} replaced by the cycle condition \eqref{eq:bfm.9b} for general
networks.
While \eqref{eq:bfm.10e} is linear in the variable $x$ the cycle condition \eqref{eq:bfm.9b}
is nonlinear in $x$.   This is the major simplification of radial networks.

When the lines are modeled by series admittances without shunt elements,
$y^m_{jk} = y^m_{kj} = 0$, then $\alpha_{jk}=\alpha_{kj}=1$ and \eqref{eq:bfm.10} and \eqref{eq:bfm.9} 
can be shown to reduce to \eqref{eq:df.1} and \eqref{eq:df.2} respectively.

\section{Equivalence and exact SOCP relaxation}

\subsection{Equivalence}

Even though BIM \eqref{eq:bim.1} and BFMs \eqref{eq:bfm.10} and \eqref{eq:bfm.9}
are defined by different sets of equations in terms of their own variables, all of them are models of the Kirchhoff's laws
and the Ohm's law 
and therefore must be related.  We now clarify the precise sense in which these mathematical models are equivalent.

Let the sets of solutions of BFMs be:
\bqn
{\mathbb X}_\text{tree}:=\{ x := (s, v, \ell, S) \in \mathbb R^{3(N+2L)}  
				\ \vert \ x \text{ satisfies } \eqref{eq:bfm.10}, v_1=1  \}~
\\
{\mathbb X}_\text{mesh}:=\{ x := (s, v, \ell, S) \in \mathbb R^{3(N+2L)}  
				\ \vert \ x \text{ satisfies } \eqref{eq:bfm.9}, v_1=1  \}
\eqn
Two sets $A$ and $B$ are said to be \emph{equivalent}, denoted  by $A\equiv B$, if
there is a bijection between them. 
All proofs in this section are deferred to Section \ref{sec:proofs}.
\begin{theorem}
\label{thm:equiv}
$\mathbb V \equiv \mathbb X_\text{mesh} \subseteq \mathbb X_\text{tree}$.  
For radial networks $\mathbb V \equiv \mathbb X_\text{mesh} = \mathbb X_\text{tree}$. 
\end{theorem}

\begin{remark}
Given the bijection between the solution sets of BIM and BFMs, any result in one model is in principle derivable in the other.   
Some results however are much easier to state or prove in one model than the other.
For instance BIM, which is widely used in transmission network problems, allows a
much cleaner formulation of the semidefinite program
(SDP) relaxation.\footnote{To be consistent, we require the cost function $f_\text{bim}$ in BIM and the cost function 
$f_\text{bfm}$ in BFMs to satisfy $f_\text{bim}(h(x)) = f_\text{bfm}(x)$ where $h$ is the bijection from a BFM to BIM.}
For radial networks, BFMs have a convenient recursive structure that allows a more efficient 
computation of power flows and leads to a useful linear approximation.
They also seem to be much more stable numerically than BIM as the network size scales up.
One should therefore freely use either model depending on
which is more convenient for the problem at hand.  
\qed
\end{remark}

\subsection{Exact SOCP relaxation for radial networks}

Consider a radial network and the following OPF problem: 
\begin{subequations}
\bq
\!\!\!\!\!\!\!\!\!\!\!
\min_x \ f(x) & \text{ s. t. } & x\in \mathbb X_\text{tree}
\label{eq:opf.1a}
\\
& & \underline s \ \leq \ s \ \leq \ \overline s
\label{eq:opf.1b}
\\
& & \underline v \ \leq \ v \ \leq \ \overline v
\label{eq:opf.1c}
\\
& & 0\ \leq \ \ell \ \leq \ \overline \ell
\label{eq:opf.1d}
\eq
\label{eq:opf.1}
\end{subequations}
We make the following assumptions:
\bi
\item[A1] The network graph $G=(N,L)$ is connected.
\item[A2] The cost function $f$ is strictly increasing in $\ell$, nondecreasing in $s$ and independent
	of $S$.
\item[A3] The OPF problem \eqref{eq:opf.1} is feasible.   
\ei
Common cost functions that satisfy A2 include those that minimize real power loss or minimize real power 
generation.

All the constraints in \eqref{eq:opf.1} are linear in $x$ except the quadratic equality \eqref{eq:bfm.10d}.
Following  \cite{Farivar-2013-BFM-TPS} we relax \eqref{eq:bfm.10d} to a second-order cone 
constraint: for $(j,k)\in L$,
\bq
\left| S_{jk}\right|^2  &  \!\!\!\!  \leq  \!\!\!\!  & v_j \, \ell_{jk}, \quad
\left| S_{kj}\right|^2    \ \leq  \ v_k \, \ell_{kj}  
\label{eq:soc.2}
\eq
Let
\bqn
{\mathbb X}_\text{soc}:=  \{ x \in \mathbb R^{3(N+2L)}  
				\ \vert \ x \text{ satisfies } \eqref{eq:bfm.10a}\eqref{eq:soc.2}
				 	\eqref{eq:bfm.10b}-\eqref{eq:bfm.10e} , v_1=1 \}
\eqn
The SOCP relaxation of OPF \eqref{eq:opf.1} is:
\begin{subequations}
\bq
\!\!\!\!\!\!\!\!\!\!\!
\min_x \ f(x) & \text{ s. t. } & x\in \mathbb X_\text{soc}
\label{eq:socp.1a}
\\
& & \underline s \ \leq \ s \ \leq \ \overline s
\label{eq:socp.1b}
\\
& & \underline v \ \leq \ v \ \leq \ \overline v
\label{eq:socp.1c}
\\
& & 0  \ \leq \ \ell \ \leq \ \overline \ell
\label{eq:socp.1d}
\eq
\label{eq:socp.1}
\end{subequations}

\vspace{-0.1in}
Consider the following conditions.
\bi
\item[C1] $\underline s_j = (-\infty, -\infty)$ for all buses $j\in N$.
\item[C2] Both $\re(\alpha_{jk})$ and $\re(\alpha_{kj})$ are strictly positive for $j\sim k$.
\ei
The next result provides a sufficient condition for exact SOCP relaxation.
\begin{theorem}
\label{thm:exactSOCP}
If C1 and C2 hold, then every optimal solution 
of the SOCP relaxation \eqref{eq:socp.1} is in $\mathbb X_\text{tree}$ and hence
optimal for OPF \eqref{eq:opf.1}.
\end{theorem}
\begin{remark}
The condition C2 generally holds in practice since shunt admittances are usually much smaller than series admittances
in magnitude (see below) and hence both $|z_{jk}^sy_{jk}^m|$ and $|z_{kj}^sy_{kj}^m|$ are strictly smaller than 1.
\end{remark}

\section{Linear approximation}
The following linear approximation of \eqref{eq:df.1}, called the LinDistFlow model, is proposed in \cite{Baran1989b} for radial networks without line shunts:
\begin{subequations}
\begin{align}
S_{jk}^\lin&=-S_{kj}^\lin=\sum\limits_{l\in\mathbb{T}_{k}}s_l ~\text{for}~j\rightarrow k,\\
v_j^\lin - v_k^\lin &=  2 \text{Re} \left(z_{jk}^* \, S_{jk}^\lin\right) ~\text{for}~j\rightarrow k,
\end{align}
\label{eq:LinDistFlow}
\end{subequations}
where $\mathbb{T}_{k}$ is the set of all the buses downstream of $k$ (including $k$).
In this section, we justify using the same model \eqref{eq:LinDistFlow} as a linear approximation
of the branch flow model \eqref{eq:bfm.10} for radial networks that include line shunts.
While one could include line shunts in a more precise nonlinear model such as \eqref{eq:bfm.10} or  \eqref{eq:bfm.9},
it is reasonable to neglect their effects in a linear approximation and use the conventional LinDistFlow model.

The model \eqref{eq:LinDistFlow} can be obtained from \eqref{eq:bfm.10} by setting $\alpha_{jk} = 1$,
$\ell_{jk} = 0$, and simplifying.  This amounts to assuming:
\begin{enumerate}
\item For each line $j\sim k$, the shunt admittances $|y_{jk}^m|$ and $|y_{kj}^m|$ are negligible so that
	$\alpha_{jk}\approx\alpha_{kj}\approx 1$.
\item The losses on both the series impedance and line shunts are negligible, and thus for $j\rightarrow k$,
	$S_{jk}\approx -s_{k}+\sum_{l:k\rightarrow l}S_{kl}$.
\end{enumerate}
These approximations can be justified as follows.
First, for typical 
 IEEE standard test cases \cite{schneider2017analytic},
the norm of $y_{jk}^m/y_{jk}^s$ is on the order of $10^{-4}{r^2}$ or smaller, implying that
$\alpha_{jk}^m$ and $\alpha_{kj}^m$ are close to $1$.
{Here $r$ is the ratio of power line length to 1 mile and typically on the order of 1 or less.}
Second, for adjacent buses $k$ and $l$, both $|V_k-V_l|/|V_k|$ and $|V_k-V_l|/|V_l|$ are typically on the order of $1\%{r}$.
In the equality
\begin{align}\label{eq:conservation_of_energy}
\nonumber
&\sum_{l:k\rightarrow l} \left( S_{kl}+|V_k-V_l|^2\big(y_{kl}^s\big)^*+|V_k|^2\big(y_{kl}^m\big)^*+|V_l|^2\big(y_{lk}^m\big)^* \right)\\
=&S_{jk}+s_{k},
\end{align}
the term $|V_k-V_l|^2\big(y_{kl}^s\big)^*$ is the loss on the series impedance, while $|V_k|^2\big(y_{kl}^m\big)^*$ and $|V_l|^2\big(y_{lk}^m\big)^*$ are the losses on shunts.
For networks without line shunts, the LinDistFlow model ignores series loss as it typically amounts to around $1\%$ of the 
branch power flow \cite{farivar2013equilibrium}.
When $y_{kl}^m/y_{kl}^s \leq 10^{-4}{r^2}$ and $|V_k-V_l|/|V_k|\approx 1\%{r}$, shunt loss $|V_k|^2\big(y_{kl}^m\big)^*$ is comparable to or 
smaller than the series loss $|V_k-V_l|^2\big(y_{kl}^s\big)^*$.
As a result, both loss terms are negligible, and \eqref{eq:conservation_of_energy} can be approximated as in the second
assumption.

\section{Conclusion}

We have proposed branch flow models \eqref{eq:bfm.10} for radial networks and \eqref{eq:bfm.9} for
general networks 
that allow nonzero shunt elements in the $\Pi$ circuit model.
We have proved (Theorem \ref{thm:equiv}) their equivalence to the bus injection model \eqref{eq:bim.1}.
For radial networks we have proved (Theorem \ref{thm:exactSOCP}) that SOCP relaxation of OPF
is exact when the injections are not lower bounded.
We have justified the use of the original LinDistFlow model as the linear approximation of the
branch flow model \eqref{eq:bfm.10} for radial networks as the effect of shunts is negligible.

\section{Appendix: Proofs}
\label{sec:proofs}

\subsection{Proof of Theorem \ref{thm:equiv}}

\begin{proof}
Let
\bqn
\mathbb{\tilde{X}} :=\{ \tilde{x} := (s, V, I, S) \in\mathbb C^{4(N+2L)} 
				\ \vert \ \tilde{x} \text{ satisfies } \eqref{eq:bfm.7}, V_1 := 1\angle 0^\circ \}
\eqn
For general networks that may contain cycles
we will prove $\mathbb V \equiv \mathbb{\tilde X}$, $\mathbb{\tilde X} \equiv \mathbb{X}_\text{mesh}$, 
and $\mathbb{X}_\text{mesh} \subseteq \mathbb{X}_\text{tree}$.
For radial networks we will prove that $\mathbb{X}_\text{mesh} = \mathbb{X}_\text{tree}$.

\begin{proof}[Proof of $\mathbb V \equiv \mathbb{\tilde X}$]
It is obvious $\mathbb V \equiv \mathbb{\tilde X}$ since given $(s, V)\in \mathbb V$, define $I$ by \eqref{eq:bfm.7c}\eqref{eq:bfm.7d}
and $S$ by \eqref{eq:bfm.7b} and the resulting $(s, V, I, S)\in \mathbb{\tilde X}$.  Conversely given $(s, V, I, S)\in \mathbb{\tilde X}$,
substituting  \eqref{eq:bfm.7b}\eqref{eq:bfm.7c}\eqref{eq:bfm.7d} into  \eqref{eq:bfm.7a} shows $(s, V)\in\mathbb V$.
As both mappings are derived from \eqref{eq:bfm.7}, it is easy to check that they are inverses of each other.
\end{proof}

\begin{proof}[Proof of $\mathbb{\tilde X} \equiv \mathbb{X}_{\rm mesh}$]
To show $\mathbb{\tilde X} \equiv \mathbb X_\text{mesh}$ fix an $\tilde x := (s, V, I, S)\in \mathbb{\tilde X}$.  
Define $(v, \ell)$ by:
\bqn
v_j \ \ := \ \ |V_j|^2  & \text{ and } & \ell_{jk} \ \ := \ \ |I_{jk}|^2
\eqn
We now show that $x := (s, v, \ell, S) \in \mathbb X_\text{mesh}$.  
Clearly \eqref{eq:bfm.10a}\eqref{eq:bfm.10d} follow from \eqref{eq:bfm.7a}\eqref{eq:bfm.7b}.
For \eqref{eq:bfm.10b} rewrite \eqref{eq:bfm.7c} as 
\bqn
V_k & = &  \alpha_{jk}\, V_j \ - \ z^s_{jk} \left( \frac{S_{jk}}{V_j} \right)^*  
\eqn
where we have substituted $I_{jk} := S_{jk}^* / V_j^*$ from \eqref{eq:bfm.7b}.
Taking the squared magnitude on both sides gives
\bqn
v_k & = & \left| \alpha_{jk} \right|^2\, v_j \ + \ \left| z^s_{jk} \right|^2\, \ell_{jk} \ - \ 2\, \text{Re}\left( \alpha_{jk}\, \left( z^s_{jk} \right)^* S_{jk} \right)
\eqn
which is \eqref{eq:bfm.10b} (recall that $z_{jk} := z_{jk}^s$).  Similarly for \eqref{eq:bfm.10c}.
From \eqref{eq:VjVk*.1} and the definitions of $\beta_{jk}(x)$ and $\beta_{kj}(x)$ in \eqref{ch:npfm; eq:Defbeta.1},
we have
\bqn
\beta_{jk}(x)  \ = \  \angle V_j - \angle V_k  \ \ = \ \  - \beta_{kj}(x)
\eqn
and hence \eqref{eq:bfm.9b} holds.  This shows $x\in\mathbb X_\text{mesh}$.
We denote this mapping from $\tilde{x}$ to $x$ as $\phi_1$.

Conversely fix an $x := (s, v, \ell, S) \in \mathbb X_\text{mesh}$.  Define $(V, I)$ as follows.  Pick a $\theta$ that satisfies
\eqref{eq:bfm.9b} with $\theta_1=0$.
\footnote{Since $[C \ -C]$ has rank $N-1$ and the all-ones vector is in its null space, such $\theta$ always uniquely exists.}
Let
\bqn
V_j & := & \sqrt{v_j}\, e^{\theta_j}, \quad j\in  N 
\eqn
Define $I_{jk}$ and $I_{kj}$ in terms of $V$ according to \eqref{eq:bfm.7c} and \eqref{eq:bfm.7d} respectively for
all lines $(j,k)\in L$.  To show that $\tilde x := (s, V, I, S)\in \mathbb{\tilde X}$ it suffices to show \eqref{eq:bfm.7b}
holds.  Since \eqref{eq:bfm.7c} holds by construction we have
\bqn
V_j I_{jk}^* & = &  y^*_{jk} (v_j - V_jV_k^*) \ + \ \left( y^m_{jk} \right)^* v_j
\\
	& = & y_{jk}^* \left( \alpha_{jk}^*\, v_j - \sqrt{v_j v_k} \, e^{\ii (\theta_j - \theta_k)} \right)
\\
	& = & y_{jk}^* \left( \left( \alpha_{jk}^*\, v_j - z_{jk}^* S_{jk} \right) - \sqrt{v_j v_k} \, e^{\ii (\theta_j - \theta_k)} \right)
		\ + \ S_{jk}
\\
	& = & S_{jk}
\eqn
as desired, if 
\bq
\left( \alpha_{jk}^*\, v_j - z_{jk}^* S_{jk} \right) & = & \sqrt{v_j v_k} \, e^{\ii (\theta_j - \theta_k)}
\label{eq:a*v-z*S}
\eq
First note that \eqref{ch:npfm; eq:Defbeta.1} and \eqref{eq:bfm.9b} imply that
both sides of \eqref{eq:a*v-z*S} have the same phase angle.
We now show that both sides have the same magnitude as well.  Indeed
\bqn
& & \left| \alpha_{jk}^*\, v_j \ - \  z_{jk}^*\, S_{jk} \right|^2
\\
 & = & 
\left| \alpha_{jk} \right|^2\, v_j^2 + \left| z_{jk} \right|^2\, |S_{jk}|^2 - 2\, \text{Re}\left( \alpha_{jk}\, z_{jk}^*\, v_j\, S_{jk} \right)
\\
& = & v_j \, v_k
\eqn
where the last equality follows from multiplying both sides of \eqref{eq:bfm.10b} by $v_j$ and then 
substituting $|S_{jk}|^2 = v_j \ell_{jk}$ from \eqref{eq:bfm.10d}.
Hence \eqref{eq:a*v-z*S} holds and $\tilde x := (s, V, I, S)$ satisfies \eqref{eq:bfm.7}
(the proof that $V_k I_{kj}^* = S_{kj}$ is similar).
{We denote the mapping from $x$ to $\tilde{x}$ as $\phi_2$. It is easy to check that $\phi_1(\phi_2(x))=x$ and $\phi_2(\phi_1(\tilde{x}))=\tilde{x}$, 
thus $\phi_1$ and $\phi_2$ are inverses of each other and there is a bijection between $\mathbb{\tilde X}$ and $\mathbb X_\text{mesh}$.}
This completes the proof that $\mathbb{\tilde X} \equiv \mathbb X_\text{mesh}$.
\end{proof}

\begin{proof}[Proof of $\mathbb X_{\rm mesh} \subseteq \mathbb X_{\rm tree}$]
Suppose $x\in\mathbb X_\text{mesh}$ and hence satisfies \eqref{eq:bfm.9}.  
The proof of \eqref{eq:a*v-z*S} also shows that
\bqn
\left( \alpha_{kj}^*\, v_k - z_{kj}^* S_{kj} \right) & = & \sqrt{v_k v_j} \, e^{\ii (\theta_k - \theta_j)}
\eqn
This together with \eqref{eq:a*v-z*S} implies \eqref{eq:bfm.10e}.
Hence $x\in\mathbb X_\text{tree}$.
\end{proof}

\begin{proof}[Proof of $\mathbb X_{\rm mesh} = \mathbb X_{\rm tree}$ for radial networks]
Suppose~$x \in \mathbb X_\text{tree}$ and hence satisfies \eqref{eq:bfm.10}.
We now show that $x$ satisfies \eqref{eq:bfm.9b} when the network is radial.   
Recall that, by construction, the first half $\tilde\beta(x)$ of $\beta(x)$ correspond to lines 
in the same directions as specified by the incidence matrix $C$.  
The condition \eqref{eq:bfm.10e} implies that $\beta(x) = [\tilde\beta^T(x) \ -\tilde\beta^T(x)]^T$.
It is well-known that $C$ has rank $N-1$ and $L = N-1$ for a (connected) radial network.  
Hence $C$ has full column rank and a solution $\theta$ to $\tilde \beta(x) = C^T\theta$
always exists and is given by:
\bqn
\theta & = & C \left( C^T C \right)^{-1}\, \tilde \beta(x)
\eqn
Therefore 
\bqn
\beta(x) \ \ = \ \ \begin{bmatrix} \tilde\beta(x) \\ - \tilde\beta(x) \end{bmatrix}
	& = & \begin{bmatrix} C^T \\ - C^T \end{bmatrix}\, \theta
\eqn
which is condition \eqref{eq:bfm.9b}.  Hence $x\in\mathbb X_\text{mesh}$.
\end{proof}

This completes the proof of Theorem \ref{thm:equiv}.
\end{proof}

\subsection{Proof of Theorem \ref{thm:exactSOCP}}
\begin{proof}
If the optimal solution $\hat{x}=(\hat{s},\hat{v},\hat{\ell},\hat{S})$ to \eqref{eq:socp.1} is not in $\mathbb X_\text{tree}$, then there must exist $j\sim k$ such that $|\hat{S}_{jk}|^2<\hat{v}_j \hat{\ell}_{jk}$ and therefore by Lemma \ref{lem:Svl},  $|\hat{S}_{kj}|^2<\hat{v}_k \hat{\ell}_{kj}$ also holds.
Following \cite{Farivar-2013-BFM-TPS} we now construct another point $\tilde{x}=(\tilde{s},\tilde{v},\tilde{\ell},\tilde{S})$ 
that $\tilde x$ is feasible and has a strictly lower cost, contradicting the optimality of $\hat x$. 
Let
\begin{align*}
\tilde{v}&=\hat{v},&&\\
\tilde{\ell}_{jk}&=\hat{\ell}_{jk}-\re(\alpha_{jk})\varepsilon,& \tilde{\ell}_{kj}&=\hat{\ell}_{kj}-\re(\alpha_{kj})\varepsilon,\\
\tilde{S}_{jk}&=\hat{S}_{jk}-z_{jk}\varepsilon/2,& \tilde{S}_{kj}&=\hat{S}_{kj}-z_{kj}\varepsilon/2,\\
\tilde{s}_j&=\hat{s}_j-z_{jk}\varepsilon/2,& \tilde{s}_k&=\hat{s}_k-z_{kj}\varepsilon/2.
\end{align*}
All the other entries of $\tilde{s},\tilde{\ell},\tilde{S}$ not listed above take the same values as in $\hat{s},\hat{\ell},\hat{S}$.

By A2, $f(\tilde{x})<f(\hat{x})$ for any $\varepsilon>0$. If there exists an $\varepsilon>0$ such that $\tilde{x}$ 
also satisfies \eqref{eq:socp.1}, then it contradicts the optimality of $\hat{x}$.
The feasibility of \eqref{eq:socp.1b} follows from  C1 and $z_{jk}=z_{kj}\geq0$, and \eqref{eq:socp.1c} naturally holds as $\tilde{v}=\hat{v}$.
Since $\ell_{jk}$ and $\ell_{kj}$ are both strictly positive, and C2 holds, there exists a sufficiently small $\varepsilon>0$ such that $\hat{\ell}\geq\tilde{\ell}\geq0$ and \eqref{eq:socp.1d} holds.

To show $\tilde{x}\in{\mathbb X}_\text{soc}$, one needs to prove $\tilde{x}$ satisfies \eqref{eq:bfm.10a}\eqref{eq:soc.2} and \eqref{eq:bfm.10b}--\eqref{eq:bfm.10e}.
For \eqref{eq:bfm.10a} at nodes $j$ and $k$, we have
\begin{align*}
\tilde{s}_j &= -z_{jk}\frac{\varepsilon}{2} + \hat{s}_j = -z_{jk}\frac{\varepsilon}{2} + \sum_{l: j\sim l} \hat{S}_{jl} = \sum_{l: j\sim l} \tilde{S}_{jl},\\
\tilde{s}_k &= -z_{kj}\frac{\varepsilon}{2} + \hat{s}_k = -z_{kj}\frac{\varepsilon}{2} + \sum_{l: k\sim l} \hat{S}_{kl} = \sum_{l: k\sim l} \tilde{S}_{kl}.
\end{align*}
Hence $\tilde{x}$ satisfies \eqref{eq:bfm.10a} at nodes $j$ and $k$.

For \eqref{eq:bfm.10b} and \eqref{eq:bfm.10c} across line $j\sim k$, we have
\begin{align*}
&|\alpha_{jk}|^2 \tilde{v}_j - \tilde{v}_k = |\alpha_{jk}|^2 \hat{v}_j - \hat{v}_k\\
= &2\, \re (\alpha_{jk}\, z_{jk}^* \hat{S}_{jk}) - | z_{jk} |^2 \hat{\ell}_{jk}\\
= &2\, \re (\alpha_{jk}\, z_{jk}^* (\tilde{S}_{jk}+z_{jk}\varepsilon/2)) - | z_{jk} |^2 (\tilde{\ell}_{jk}+\re(\alpha_{jk}) \epsilon)\\
= &2\, \re (\alpha_{jk}\, z_{jk}^* \tilde{S}_{jk}) - | z_{jk} |^2 \tilde{\ell}_{jk}.
\end{align*}
Likewise, $|\alpha_{kj}|^2 \tilde{v}_k - \tilde{v}_j = 2\, \re (\alpha_{kj}\, z_{kj}^* \tilde{S}_{kj}) - | z_{kj} |^2 \tilde{\ell}_{kj}$ also holds. Hence $\tilde{x}$ satisfies \eqref{eq:bfm.10b} and \eqref{eq:bfm.10c} across line $j\sim k$.

For \eqref{eq:bfm.10e} across line $j\sim k$, we have
\begin{align*}
&\alpha_{jk}^* \tilde{v}_j -  z_{jk}^* \tilde{S}_{jk}  = \alpha_{jk}^* \hat{v}_j -  z_{jk}^* \left(\hat{S}_{jk}-z_{jk}\frac{\varepsilon}{2} \right)\\
=&(\alpha_{jk}^* \hat{v}_j -  z_{jk}^* \hat{S}_{jk}) + |z_{jk}|^2\frac{\varepsilon}{2}\\
=&(\alpha_{kj}^* \hat{v}_k -  z_{kj}^* \hat{S}_{kj})^* + |z_{kj}|^2\frac{\varepsilon}{2}\\
=&\left(\alpha_{kj}^* \hat{v}_k -  z_{kj}^* \left(\hat{S}_{kj}-z_{kj}\frac{\varepsilon}{2}\right)\right)^* \\
=&( \alpha_{kj}^* \tilde{v}_k - z_{kj}^* \tilde{S}_{kj} )^*.
\end{align*}
Hence $\tilde{x}$ satisfies \eqref{eq:bfm.10e} across line $j\sim k$.

For \eqref{eq:soc.2} across line $j\sim k$, since $| \hat{S}_{jk}|^2   <  \hat{v}_j \, \hat{\ell}_{jk}$, we have
\begin{align*}
&\tilde{v}_j \, \tilde{\ell}_{jk} - | \tilde{S}_{jk}|^2 =  \hat{v}_j \, \left(\hat{\ell}_{kj}-\re(\alpha_{kj})\varepsilon\right)- \left|\hat{S}_{jk}-z_{jk}\frac{\varepsilon}{2}\right|^2\\
=&\hat{v}_j \, \hat{\ell}_{jk} - | \hat{S}_{jk}|^2 - \varepsilon\left(\re(\alpha_{kj}\hat{v}_j-\hat{S}_{jk}z_{jk}^*) + \frac{|z_{jk}|^2}{4}\varepsilon\right)
\end{align*}
is positive for sufficiently small $\epsilon$. Likewise, $| \tilde{S}_{jk}|^2   <  \tilde{v}_j \, \tilde{\ell}_{jk}$. Hence $\tilde{x}$ satisfies \eqref{eq:soc.2} across line $j\sim k$.
\end{proof}

\bibliographystyle{unsrt}
\bibliography{PowerRef-201202}

\end{document}